\documentclass[11pt,a4paper]{amsart}
\usepackage{amsfonts}
\usepackage{amsthm}
\usepackage{amsmath}
\usepackage{amscd}
\usepackage[latin2]{inputenc}
\usepackage{t1enc}
\usepackage[mathscr]{eucal}
\usepackage{indentfirst}
\usepackage{graphicx}
\usepackage{graphics,color}
\usepackage{pict2e}
\usepackage{epic}
\usepackage{epstopdf}

\usepackage{hyperref}

\newtheorem{theorem}{Theorem}[section]
\newtheorem*{theorem*}{Theorem}
\newtheorem{proposition}[theorem]{Proposition}

\newtheorem{remark}[theorem]{Remark}

\begin{document}

\title[Alexandrov--Fenchel inequalities in hyperbolic space]{Weighted Alexandrov-Fenchel inequalities in hyperbolic space and a conjecture of Ge, Wang and Wu}

\author[F. Gir\~ao]{Frederico Gir\~ao}

%\address{Universidade Federal do Cear\'{a}\\ departamento de matem\'{a}tica\\campus do pici\\av. humberto monte,s/n,bloco 914,60455-760\\fortaleza/ce\\brazil}

\author[D. Pinheiro]{Diego Pinheiro}

\author[N. M. Pinheiro]{Neilha M. Pinheiro}

%\address{Universidade Federal do Cear\'{a}\\ departamento de matem\'{a}tica\\campus do pici\\av. humberto monte,s/n,bloco 914,60455-760\\fortaleza/ce\\brazil}

\author[D. Rodrigues]{Diego Rodrigues}

\email{fred@mat.ufc.br}
\email{diegodsp01@gmail.com}
\email{neilhamat@gmail.com}
\email{diego.sousa.ismart@gmail.com}

\thanks{Frederico Gir\~ao was partially supported by CNPq, grant number 306196/2016-6 and by FUNCAP/CNPq/PRONEX, grant number 00068.01.00/15. This study was financed in part by the Coordena\c c\~ao de Aperfei\c coamento de Pessoal de N\'{\i}vel Superior - Brasil (CAPES) - Finance Code 001.}

%Neilha M. Pinheiro, Diego Pinheiro and Diego Rodrigues were partially supported by doctoral scholarships from CAPES}



\begin{abstract} 
We consider a conjecture made by Ge, Wang and Wu regarding weighted Alexandrov--Fenchel inequalities for horospherically convex hypersurfaces in hyperbolic space
(a bound, for some physically motivated weight function, of the weighted integral of the $k^{\mathrm{th}}$ mean curvature in terms of the area of the hypersurface). We prove an inequality very similar to the conjectured one. Moreover, when $k$ is zero and the ambient space has dimension three, we give a counterexample to the conjectured inequality.
\end{abstract}

\maketitle

\section{Introduction}
Let $\Sigma$ be a convex hypersurface in $\mathbb{R}^n$, $n \geq 3$. The Alexandrov--Fenchel inequalities \cite{Ale37, Ale38} state that
\begin{equation} \label{A-F inequalities}
\left( \frac{1}{\omega_{n-1}} \int_\Sigma H_k \, d\Sigma \right)^{n-k} \geq   \left( \frac{1}{\omega_{n-1}}\int_\Sigma H_{k-1} \, d\Sigma \right)^{n-k-1},
\end{equation}
for $k = 1, \ldots, n-1$, where $\omega_{n-1}$ is the area of the unit sphere $\mathbb{S}^{n-1} \subset \mathbb{R}^{n}$ and $H_k$ is the normalized $k^{\mathrm{th}}$ mean curvature of $\Sigma$, that is,
$$
H_k = \frac{1}{C_{n-1}^k} \sigma_k,
$$
$k = 0, 1, \ldots, n-1$, with $\sigma_k$ being the $k^{\mathrm{th}}$ elementary symmetric function of the principal curvature vector $\lambda = (\lambda_1, \ldots, \lambda_{n-1})$. Moreover, the equality holds if and only if $\Sigma$ is a round sphere. In \cite{Guan-Li}, using a certain inverse curvature flow, Guan and Li showed that (\ref{A-F inequalities}) still hods for any $\Sigma$ which is star-shaped and $k$-convex (which means that $\sigma_i(\lambda) \geq 0$ for $i = 0, 1, \ldots, k$).

The $k=1$ case of (\ref{A-F inequalities}), namely,
$$
\int_\Sigma H_1 \, d\Sigma \geq  \omega_{n-1} \left( \frac{|\Sigma|}{\omega_{n-1}} \right)^{\frac{n-2}{n-1}},
$$
where $|\Sigma|$ is the area of $\Sigma$, is a key step in the proof of the Penrose inequality for graphs, given by Lam in \cite{Lam} (see also \cite{dLG1} and \cite{MV}). More generally, the cases of (\ref{A-F inequalities}) for which $k$ is odd were used in a crucial way to establish, for graphs, versions of the Penrose inequality in the context of the so called Gauss--Bonnet--Chern mass \cite{GWW1} (see also \cite{LWX} and \cite{dSG}).

Let us now consider the hyperbolic $n$-space $\mathbb{H}^n$ to be the ambient space. We will work with two models of $\mathbb{H}^n$: the warped product model and the Poincar\'e ball model. The former consists of $\mathbb{R}_+ \times \mathbb{S}^{n-1}$ endowed with the metric
$$
dr^2 + (\sinh^2 r)h,
$$
where $h$ is the round metric on the unit sphere $\mathbb{S}^{n-1} \subset \mathbb{R}^n$. The later consists of the unit ball 
$$\mathbb{B}^n = \lbrace x \in \mathbb{R}^n ; \ |x| \leq 1 \rbrace$$
endowed with the metric
$$
\left( \frac{2}{1 - |x|^2} \right)^2 \delta,
$$
where $| \ |$ denotes the Euclidean norm and $\delta$ denotes the Euclidean metric.

A hypersurface $\Sigma$ in $\mathbb{H}^n$ is said to be \emph{star-shaped} if it can be written as a graph over a geodesic sphere centered at the origin. We say that $\Sigma$ is \emph{strictly mean-convex} if its mean curvature $H_1$ is positive everywhere. Also, $\Sigma$ is said to be \emph{horospherically convex} if all of its principal curvatures are greater than or equal to $1$.

We consider the function $\rho: \mathbb{H}^n \to \mathbb{R}$ which in the warped product model is given by
$$
\rho = \cosh{r}.
$$
When working with the Poincar\'e model, the function $\rho$ has the expression
$$
\rho = \frac{1 + |x|^2}{1 - |x|^2}.
$$

 We consider also the \emph{support function} $p: \Sigma \to \mathbb{R}$, which is defined by
$$
    p = \langle D \rho, \xi \rangle,
$$
where $\xi$ is the outward unit normal vector to $\Sigma$ and where $\langle \ , \ \rangle$ denotes the hyperbolic metric and $D$ denotes its Levi-Civita connection.

In \cite{dLG2}, de Lima together with the first named author showed the following Alexandrov--Fenchel-type inequality: if $\Sigma$ is a star-shaped and strictly mean-convex hy\-per\-sur\-face in $\mathbb{H}^n$, $n \geq 3$, then
\begin{equation} \label{dLG inequality}
\int_\Sigma \rho H_1 \, d\Sigma \geq  \omega_{n-1} \left[ \left( \frac{|\Sigma|}{\omega_{n-1}} \right)^{\frac{n-2}{n-1}} + \left( \frac{|\Sigma|}{\omega_{n-1}} \right)^{\frac{n}{n-1}} \right],
\end{equation}
with the equality occurring if and only if $\Sigma$ is a geodesic sphere centered at the origin. The proof uses, among other ingredients, two monotone quantities along the inverse mean curvature flow (IMCF) and an inequality due to Brendle, Hung and Wang \cite{BHW}. Inequality (\ref{dLG inequality}) was conjectured by Dahl, Gicquaud and Sakovich in \cite{DGS}, where they found an explicit formula for the mass of an asymptotically hyperbolic graph; (\ref{dLG inequality}) was then the only thing left to show in order to proved the Penrose inequality in this context.

In \cite{GWW2}, Ge, Wang and Wu defined the Gauss--Bonnet--Chern mass for asymptotically hyperbolic manifolds.
%(an analog of the Gauss--Bonnet--Chern mass for asymptotic flat manifolds). 
In order to establish, in this context, the Penrose inequality for graphs, they showed, for odd $k$, the following weighted Alexandrov--Fenchel-type inequality: if $\Sigma$ is a horospherically convex hypersurface in $\mathbb{H}^n$, then it holds
\begin{equation} \label{GWW inequality}
\int_\Sigma \rho H_k \, d\Sigma \geq
\omega_{n-1} \left[ \left( \frac{|\Sigma|}{\omega_{n-1}} \right)^{\frac{2n}{(k+1)(n-1)}} + \left( \frac{|\Sigma|}{\omega_{n-1}} \right)^{\frac{2(n-k-1)}{(k+1)(n-1)}} \right]^{\frac{k+1}{2}},
\end{equation}
with the equality occurring if and only if $\Sigma$ is a geodesic sphere centered at the origin. They accomplished this by an induction argument (from $j$ to $j+2$), with the base case being inequality (\ref{dLG inequality}). 

Also in \cite{GWW2} it was conjectured that (\ref{GWW inequality}) holds for even values of $k$ as well. They remarked that the induction argument (from $j$ to $j+2$) still works in this case. Thus, it would be enough to show the validity of (\ref{GWW inequality}) for $k=0$, that is,
\begin{equation} \label{conjecture}
    \int_\Sigma \rho \, d\Sigma \geq \omega_{n-1} \left[ \left( \frac{|\Sigma|}{\omega_{n-1}} \right)^{\frac{2n}{n-1}} + \left( \frac{|\Sigma|}{\omega_{n-1}} \right)^{2} \right]^{\frac{1}{2}}.
\end{equation}

Now let's state the main results of this paper. Our first main result shows the existence of a counterexample to (\ref{conjecture}) when $n=3$. 

\begin{theorem} \label{first theorem}
There exists a horospherically convex hypersurface $\Gamma$ in $\mathbb{H}^3$ such that
$$
\int_\Gamma \rho \, d\Gamma < \omega_{n-1} \left[ \left( \frac{|\Gamma|}{\omega_{n-1}} \right)^{\frac{2n}{n-1}} + \left( \frac{|\Gamma|}{\omega_{n-1}} \right)^{2} \right]^{\frac{1}{2}}.
$$
\end{theorem}

Our second main result is an inequality very similar to (\ref{conjecture}). The precise statement is the following:

\begin{theorem} \label{second theorem}
Let $\Sigma$ be a star-shaped hypersurface in $\mathbb{H}^n$ satisfying $$H_1 \geq 1.$$ It holds that
\begin{equation} \label{weak inequality}
    \int_\Sigma \rho \, d\Sigma > \omega_{n-1} \left[ \left( \frac{n-1}{n} \right)^2 \left( \frac{|\Sigma|}{\omega_{n-1}} \right)^{\frac{2n}{n-1}} + \left( \frac{|\Sigma|}{\omega_{n-1}} \right)^{2} \right]^{\frac{1}{2}}.
\end{equation} 
\end{theorem}

We now state our third and final main result, which is an inequality very similar to (\ref{GWW inequality}).
%Using an induction argument similar to the one in \cite{GWW2} but with (\ref{weak inequality}) as the base case, we were able to show the following:

\begin{theorem} \label{third theorem}
If $\Sigma$ is a horospherically convex hypersurface in $\mathbb{H}^n$ and $k \in \lbrace 0,1, \ldots, n-1 \rbrace$ is even, then it holds
%\begin{equation} 
%\label{weak GWW inequality}
$$
\int_\Sigma \rho H_k \, d\Sigma >
\omega_{n-1} \left[ \left( \frac{n-1}{n} \right)^2 \left( \frac{|\Sigma|}{\omega_{n-1}} \right)^{\frac{2n}{(k+1)(n-1)}} + \left( \frac{|\Sigma|}{\omega_{n-1}} \right)^{\frac{2(n-k-1)}{(k+1)(n-1)}} \right]^{\frac{k+1}{2}}.
$$
%\end{equation}
\end{theorem}

\section{Variation formulae}
Let $\psi_0: \Sigma \to \mathbb{H}^n$ be a closed, isometrically immersed oriented hypersurface.
%with outward pointing unit normal $\xi$. 
We consider a one-parameter family $\Psi(t,\cdot): \Sigma \to \mathbb{H}^n$ of isometrically immersed hypersurfaces evolving according to
\begin{equation} \label{extrinsic flow}
    \frac{\partial \Psi}{\partial t} = F\xi,
\end{equation}
with $\Psi(0, \cdot) = \psi_0$, where $\xi$ is the outward unit normal to $\Psi(t,\cdot): \Sigma \to \mathbb{H}^n$ and $F$ is a general speed function.

\begin{proposition}
Along the flow (\ref{extrinsic flow}), the following evolution equations hold:
\begin{itemize}
%\item[] The unit normal evolves as
%$$
%\frac{\partial \xi}{\partial t} = \nabla F;
%$$
\item[] The area element $d\Sigma$ evolves as
\begin{equation} \label{evolution area element}
\frac{\partial}{\partial t} d\Sigma = F \sigma_1 d\Sigma.
\end{equation}
In particular, $|\Sigma|$, the area of $\Sigma$, evolves as
\begin{equation} \label{evolution area}
\frac{d}{dt} |\Sigma| = \int_\Sigma F \sigma_1 \, d\Sigma.
\end{equation}
\item[] The function $\rho$ evolves as
\begin{equation} \label{evolution rho}
    \frac{\partial \rho}{\partial t} = pF.
\end{equation}
\end{itemize}
\end{proposition}

\begin{proof}
Formulas (\ref{evolution area element}) and (\ref{evolution area}) are well known (see, for example, \cite{huisken}). Equation (\ref{evolution rho}) is proven, for example, in \cite{dLG2} (Proposition 3.2).
\end{proof}

Of particular interest to us is the case $F=-p$, so that $\Sigma$ evolves according to
\begin{equation} \label{sff}
\frac{\partial \Psi}{\partial t} = -p\xi.
\end{equation}
This flow will be called \emph{support function flow} (SFF). 
%It will be shown later (Proposition \ref{sff exists for all time}) that this flow exists for all time.

From now on we use the Poincar\'e ball model to represent the hyperbolic space.

Next we consider, for each $t \in [0,\infty)$, the hypersurface $\varphi_t: \Sigma \to \mathbb{B}^n$ defined by
\begin{equation} \label{varphi_t}
\varphi_t = e^{-t} \psi_0.
\end{equation}
Notice that if $\Phi: \mathbb{R} \times \Sigma \to \mathbb{B}^n$ is defined by
$$ \Phi (t,p) = \varphi_t(p), $$
then it satisfies the differential equation
\begin{equation} \label{homothetic flow}
\frac{\partial \Phi}{\partial t} = -\Phi.
\end{equation}

We have that (\ref{varphi_t}) defines, for any hypersurface $\Sigma_0$ in $\mathbb{H}^n$ a $1$-parameter family $\lbrace \Sigma_t \rbrace_{t \geq 0}$ of hypersurfaces in $\mathbb{H}^n$. Whenever no confusion arises, we will write only $\Sigma$ to denote $\Sigma_t$.

\begin{remark}
{\rm Notice that, from the Euclidean point of view (that is, by endowing $\mathbb{B}^n$ with the Euclidean metric $\delta$), $\Sigma_t$ is just the image of $\Sigma_0$ under the homothety of center in the origin and ratio $e^{-t}$.}
\end{remark}

\begin{proposition}\label{sff exists for all time}
The flow (\ref{sff}) exists for all time.
\end{proposition}
\begin{proof}
By the same argument given in Proposition 1.3.4 of  \cite{mantegazza}, as long as the flow (\ref{homothetic flow}) exists, then the flow
\begin{equation} \label{nhf}
\frac{\partial \Phi}{\partial t} = -\langle \Phi, \xi \rangle \xi
\end{equation}
also exists. Since (\ref{homothetic flow}) exists for all time, (\ref{nhf}) also exists for all time.
However, when working with the ball model, a simple computation shows that 
$$
D \rho = X,
$$
where $X$ is the vector field that associates to each $x \in \mathbb{B}^n$ the vector $x$. Thus, $\langle \Phi, \xi \rangle$ is the support function and the flow (\ref{nhf}) coincides with the flow (\ref{sff}).
\end{proof}

\begin{remark}
{\rm The argument given in Proposition 1.3.4 of \cite{mantegazza} actually shows that the flows (\ref{sff}) and (\ref{homothetic flow}) are, up to reparametrization, the same flow. For this reason, we will abuse notation and also denote by $\lbrace \Sigma_t \rbrace_{t \geq 0}$ the 1-parameter family of hypersurfaces defined by (\ref{sff}). Again, whenever no confusion arises, we will write only $\Sigma$ to denote $\Sigma_t$.}
\end{remark}

For a hypersurface $\Sigma$ in $\mathbb{H}^n$ we define the quantity $I(\Sigma)$ by
%\begin{equation} \label{I}
$$
    \mathcal{I} = \int_\Sigma \rho \, d\Sigma.
$$
%\end{equation}

\begin{proposition}
Along the flow (\ref{sff}) the following evolution equations hold:
\begin{itemize}
\item[] The area $|\Sigma|$ evolves as
\begin{equation} \label{evolution area SFF}
    \frac{d}{dt}|\Sigma| = -(n-1) \mathcal{I}.
\end{equation}
\item[] The quantity $\mathcal{I}$ evolves as
\begin{equation} \label{evolution total rho SFF}
\frac{d \mathcal{I}}{dt} = |\Sigma| - n \int_\Sigma \rho^2 \, d\Sigma.
\end{equation}
\end{itemize}
\end{proposition}
\begin{proof}
We have
\begin{equation} \label{first minkowski}
\Delta_\Sigma \rho = (n-1)\rho -p \sigma_1
\end{equation}
and
\begin{equation} \label{eq 7.3 GWW}
\rho^2 = 1 + p^2 + \langle \nabla \rho, \nabla \rho \rangle.    
\end{equation}
Identities (\ref{first minkowski}) and (\ref{eq 7.3 GWW}) are proven, for example, in \cite{GWW2} (Lemma 7.1). Integrating (\ref{first minkowski}) we get
\begin{equation} \label{integrated first minkowski}
%(n-1) \int_\Sigma \rho \, d\Sigma 
(n-1)\mathcal{I} = \int_\Sigma p \sigma_1 \, d\Sigma.    
\end{equation}
Equation (\ref{evolution area SFF}) follows from (\ref{evolution area}) and (\ref{integrated first minkowski}). Multiplying (\ref{first minkowski}) by $\rho$ and integrating yields 
\begin{equation} \label{ibp}
-\int_\Sigma \langle \nabla \rho, \nabla \rho \rangle \, d\Sigma = (n-1) \int_\Sigma \rho^2 \, d\Sigma - \int_\Sigma \rho p \sigma_1 \, d\Sigma.    
\end{equation}
Using (\ref{evolution rho}), (\ref{evolution area element}), (\ref{ibp}) and (\ref{eq 7.3 GWW}) we find
\begin{align*}
\frac{d\mathcal{I}}{dt}
& = \int_\Sigma \frac{\partial \rho}{\partial t} \, d\Sigma + \int_\Sigma \rho \frac{\partial}{\partial t} d\Sigma \\
& = - \int_\Sigma p^2 \, d\Sigma - \int_\Sigma \rho p \sigma_1 \, d\Sigma \\
& = - \int_\Sigma p^2 \, d\Sigma - \int_\Sigma \langle \nabla \rho, \nabla \rho \rangle \, d\Sigma - (n-1) \int_\Sigma \rho^2 \, d\Sigma \\
& = |\Sigma| - n\int_\Sigma \rho^2 \, d\Sigma,
\end{align*}
as wished.
\end{proof}

For a hypersurface $\Sigma$ in $\mathbb{H}^n$, define the quantity $ \mathcal{P}(\Sigma)$ by
\begin{equation} \label{P}
    \mathcal{P} = \left[ \omega_{n-1} \left( \frac{|\Sigma|}{\omega_{n-1}}\right)^{\frac{n}{n-1}} \right]^{-2} \left[ 
    %\left( \int_\Sigma \rho \, d\Sigma \right)^2
   \mathcal{I}^2 - |\Sigma|^2 \right].
\end{equation}

\begin{proposition} \label{prop P is nonincreasing}
Along the flow (\ref{sff}) it holds
%\begin{equation} \label{P is nonincreasing}
$$
    \frac{d \mathcal{P}}{dt} \leq 0.
$$
%\end{equation}
Moreover, the equality holds at $t$ if and only if $\Sigma_t$ is a geodesic sphere centered at the origin.
\end{proposition}
\begin{proof}
First, note that H\"older's inequality applied to (\ref{evolution total rho SFF}) gives
\begin{equation} \label{holder}
\frac{d\mathcal{I}}{dt} \leq |\Sigma| - n\dfrac{\mathcal{I}^2}{|\Sigma|},
\end{equation}
with the equality holding if and only if $\rho$ is constant on $\Sigma$, that is, if and only if $\Sigma$ is a geodesic sphere centered at the origin.

Now, a straightforward computation together with (\ref{holder}), (\ref{evolution area SFF}) and (\ref{evolution total rho SFF}) yields 
\begin{align}
    \dfrac{d\mathcal{P}}{dt} 
    & = \dfrac{ \left[  2\mathcal{I} \dfrac{ d \mathcal{I}}{dt} -2|\Sigma| \dfrac{d}{dt}|\Sigma| -  \dfrac{2n}{n-1}\left( \mathcal{I}^2 - |\Sigma|^2 \right)|\Sigma|^{-1} \dfrac{d}{dt}|\Sigma| \right]}{\left[ \omega \left( \dfrac{|\Sigma|}{\omega} \right)^{\frac{n}{n-1}} \right]^{2}} \nonumber \\
    & \leq
    \dfrac{ \left[  2\mathcal{I} \left( |\Sigma| - n\dfrac{\mathcal{I}^2}{|\Sigma|} \right) -2|\Sigma| \dfrac{d}{dt}|\Sigma| -  \dfrac{2n}{n-1}\left( \mathcal{I}^2 - |\Sigma|^2 \right)|\Sigma|^{-1} \dfrac{d}{dt}|\Sigma| \right]}{\left[ \omega \left( \dfrac{|\Sigma|}{\omega} \right)^{\frac{n}{n-1}} \right]^{2}} \label{equality derivative P} \\
    & = 0. \nonumber
\end{align}
The equality holds in (\ref{equality derivative P}) if and only if it also holds in (\ref{holder}), which occurs if and only if $\Sigma$ is a geodesic sphere centered at the origin.
\end{proof}

\begin{proposition} \label{prop P noinicreasing 2}
Along the flow (\ref{homothetic flow}) the quantity $\mathcal{P}$ defined by (\ref{P}) satisfies
%\begin{equation} \label{P is nonincreasing 2}
$$
    \frac{d \mathcal{P}}{dt} \leq 0.
$$
%\end{equation}
Furthermore, the equality holds at $t$ if and only if $\Sigma_t$ is a geodesic sphere centered at the origin.
\end{proposition}
\begin{proof}
In order to compute the variation of $\mathcal{P}$ along (\ref{homothetic flow}), we can disregard tangential motions, that is, instead of the flow (\ref{homothetic flow}), we can consider the flow (\ref{nhf}) which, as argued in the proof of Proposition \ref{sff exists for all time}, coincides with the flow (\ref{sff}).
Hence, the proposition follows from Proposition \ref{prop P is nonincreasing}.
\end{proof}

A hypersurface $\Sigma$ in $\mathbb{H}^n$ can also be seen as an Euclidean hypersurface (just endow $\mathbb{B}^n$ with the Euclidean metric $\delta$).
%We will use the notation $\Sigma^\delta$ to indicate that $\Sigma$ is being seen as a Euclidean hypersurface.

For an Euclidean hypersurface $\Sigma$ we define the quantity $\mathcal{Q}(\Sigma)$ by
%\begin{equation} \label{Q}

$$    \mathcal{Q}(\Sigma) = \left[ \omega_{n-1} \left( \frac{|\Sigma|_\delta}{\omega_{n-1}}  \right)^{\frac{n+1}{n-1}} \right]^{-1} \int_{\Sigma} |x|^2 \, (d\Sigma)_\delta,
$$
%\end{equation}
where $|\Sigma|_\delta$ and $(d\Sigma)_\delta$ are the area and the area element of $\Sigma$ with respect to the metric induced by the Euclidean metric.

The next proposition relates the quantities $\mathcal{P}$ and $\mathcal{Q}$.

\begin{proposition} \label{lim P to Q}It holds
\begin{equation} \label{P to Q}
 \lim_{t \to \infty} \mathcal{P}(\Sigma_t) = \mathcal{Q}(\Sigma_0).
 \end{equation}
\end{proposition}
\begin{proof}
First, note that since $\mathcal{P}(\Sigma_t)$ is decreasing and bounded below (by $0$), the limit on the left hand side of (\ref{P to Q}), in fact, exists.

Also, since the quantities
$$
\mathcal{I}^2 - |\Sigma|^2
%\left( \int_\Sigma \rho \, d\Sigma \right)^2 - |\Sigma|^2
$$
and
$$
\omega_{n-1} \left( \frac{|\Sigma|}{\omega_{n-1}} \right)^{\frac{2n}{n-1}}
$$
converge to $0$, l'H\^opital's rule together with (\ref{evolution area SFF}), (\ref{evolution total rho SFF}) and a straightforward computation give 
$$
\lim_{t \to \infty} \mathcal{P}(\Sigma) = \lim_{t \to 0} \left[ \omega_{n-1}\left( \frac{|\Sigma|}{\omega_{n-1}} \right)^{\frac{n+1}{n-1}} \right]^{-1}\int_\Sigma (\rho^2 - 1) \, d\Sigma.
$$

Let $\varepsilon > 0$ be given. Take $\eta > 0$ such that $|x| < \eta$ implies 
$$
\left| \left( \frac{1}{1-|x|^2} \right)^{n+1} - 1 \right| < \varepsilon \ \ \text{and} \ \ \left| \left( \frac{1}{1-|x|^2} \right)^{n-1} - 1 \right| < \varepsilon. $$
Using that 
$$
\rho = \frac{1 + |x|^2}{1 - |x|^2}
$$
and that
$$ 
d\Sigma = \left( \frac{2}{1-|x|^2} \right)^{n-1}\left( d\Sigma \right)_\delta
$$
we have, for each $\Sigma$ contained in $\lbrace x \in \mathbb{B}^n; \ |x| < \eta \rbrace$, that
$$
\left| \frac{\displaystyle\int_\Sigma (\rho^2 - 1) \, d\Sigma}{2^{n+1} \displaystyle\int_{\Sigma} |x|^2 \, (d\Sigma)_\delta} -1 \right|
\leq \frac{\displaystyle\int_\Sigma |x|^2 \left| \left( \frac{1}{1-|x|^2} \right)^{n+1} - 1 \right| \, (d\Sigma)_\delta}{\displaystyle\int_{\Sigma} |x|^2 \, (d\Sigma)_\delta}
< \varepsilon
$$
and
$$
\left| \frac{\left( \frac{|\Sigma|}{\omega_{n-1}} \right)}{2^{n-1}\left( \frac{|\Sigma|_\delta}{\omega_{n-1}} \right)} - 1 \right|
\leq \frac{\displaystyle\int_\Sigma \left| \left( \frac{1}{1-|x|^2} \right)^{n-1} - 1 \right| \, (d\Sigma)_\delta}{|\Sigma|_\delta} < \varepsilon.
$$
Hence,
\begin{equation} \label{limits}
\lim_{t \to \infty} \frac{\int_\Sigma (\rho^2 - 1) \, d\Sigma}{2^{n+1} \int_\Sigma |x|^2 \, (d\Sigma)_\delta} = 1 \ \ \text{and} \ \ 
\lim_{t \to \infty} \frac{\left( \frac{|\Sigma|}{\omega_{n-1}} \right)}{2^{n-1}\left( \frac{|\Sigma|_\delta}{\omega_{n-1}} \right)} = 1.
\end{equation}
Using (\ref{limits}) and the scale invariance of the quantity $\mathcal{Q}$ we have
\begin{align*}
    \lim_{t \to \infty} \frac{\mathcal{P}(\Sigma_t)}{\mathcal{Q}(\Sigma_0)}
    & = \lim_{t \to \infty} \frac{\mathcal{P}(\Sigma_t)}{\mathcal{Q}(\Sigma_t)} \\
    & = \lim_{t \to \infty} \left[ \left( \frac{\int_\Sigma (\rho^2 - 1) \, d\Sigma }{\omega_{n-1} \left( \frac{|\Sigma|}{\omega_{n-1}} \right)^{\frac{n+1}{n-1}}} \right) \left( \frac{\int_\Sigma |x|^2\, (d\Sigma)_\delta}{\omega_{n-1}\left( \frac{|\Sigma|}{\omega_{n-1}} \right)^{\frac{n+1}{n-1}}} \right)^{-1}\right] \\
    & = \left( \lim_{t \to \infty} \frac{\int_\Sigma (\rho^2 - 1) \, d\Sigma}{2^{n+1} \int_{\Sigma} |x|^2 \, (d\Sigma)_\delta} \right)
        \left( \lim_{t \to \infty} \frac{\left( \frac{|\Sigma|}{\omega_{n-1}} \right)}{2^{n-1}\left( \frac{|\Sigma|_\delta}{\omega_{n-1}} \right)} \right)^{-\frac{n+1}{n-1}} \\
    & = 1.
\end{align*}
\end{proof}

The following two propositions relate the geometry of $\Sigma$ as a hypersurface in $\mathbb{H}^n$ with the geometry of $\Sigma$ as an Euclidean hypersurface.

\begin{proposition} \label{mean curvatures}
Let $\psi: \Sigma \to \mathbb{B}^n$ be so that, as a hypersurface in $\mathbb{H}^n$, its mean curvature satisfies $H_1 \geq 1$. Then, as an Euclidean hypersurface, $\Sigma$ is mean-convex.
\end{proposition}
\begin{proof}
In Poincar\'e's model for $\mathbb{H}^n$, the hyperbolic metric is given by $$\langle \ , \ \rangle = \phi^2 \delta,$$ where
\begin{equation} \label{phi}
\phi = \frac{2}{1 - |x|^2}.
\end{equation}
In particular, since $\langle \xi , \xi \rangle = 1,$ it follows that
$$ \delta(\phi\xi,\phi\xi) = 1,$$
that is,
\begin{equation} \label{norma de xi}
|\phi\xi| = 1.    
\end{equation}
The well known formula for the mean curvature under a conformal change of metric gives
$$
H_1 = \phi^{-1} H_1^\delta + \phi^{-1} \xi (\phi),
$$
where $H_1^\delta$ denotes the mean curvature of $\Sigma$ as an Euclidean hypersurface.
Using that 
\begin{equation} \label{derivative of phi}
\xi (\phi) = \phi^2 \delta (\xi,\psi),
\end{equation}
we find
$$
H_1^\delta =  \phi \left( H_1 - \phi\delta(\xi,\psi) \right) > 0
$$
since $H_1 \geq 1$ and, by Cauchy's inequality together with (\ref{norma de xi}),
$$
\phi\delta (\xi, \psi) \leq \phi |\xi| \cdot |\psi| = |\phi \xi|\cdot |\psi| = |\psi| < 1.
$$ 
\end{proof}

\begin{proposition} \label{c to hc}
Let $\psi_0: \Sigma \to \mathbb{B}^n$ be such that, as an Euclidean hypersurface, $\Sigma$ is strictly convex. Then, there exists $T \in [0,\infty)$ for which
$\varphi_t:\Sigma \to \mathbb{B}^n$ given by $\varphi_t = e^{-t}\psi_0$ is horospherically convex for each $t \geq T$.
\end{proposition}
\begin{proof}
Let $b$ and $b^\delta$ be the second fundamental forms of $\Sigma$ and $\Sigma^\delta$, respectively. A well known formula in conformal geometry gives
$$
b = \phi b^\delta + \phi\xi(\phi) \delta,
$$
where $\phi$ is defined by (\ref{phi}). Together with (\ref{derivative of phi}), this gives
$$
b = \phi b^\delta + \phi^3 \delta (\xi, \psi) \delta.
$$
Thus, for any tangent vector $v$ we have
$$
b(v,v) = \phi b^\delta (v,v) + \phi^3 \delta (\xi, \psi) \delta (v,v).
$$
Hence, using the convexity of $\Sigma$, we find
$$
b(v,v) \geq \phi b^\delta (v,v).
$$
%Dividing both sides by $\phi^2\delta(v,v)$ we get
%$$
%\frac{b(v,v)}{\phi^2\delta(v,v)} \geq \phi^{-1} \frac{b^\delta (v,v)}{\delta(v,v)}.
%$$
Now, let $b_t$ and $b_t^\delta$ be the second fundamental forms of $\Sigma_t$ and $\Sigma_t^\delta$, respectively. The previous inequality gives
\begin{equation} \label{inequality sff 1}
b_t(v,v) \geq \phi_t^{-1}b_t^\delta (\phi_t v,\phi_t v).
\end{equation}
%$$
%\frac{b_t(v,v)}{\phi_t^2\delta(v,v)} \geq \phi_t^{-1} \frac{b_t^\delta (v,v)}{\delta(v,v)}.
%$$
%where 
%$$
%\phi_t = \frac{2}{1 - %|e^{-t}x|^2}.
%$$
Also, since $b^\delta_t = e^{-t} b^\delta$, we have
\begin{equation} \label{inequality sff 2}
b^\delta_t (\phi_t v, \phi_t v) = e^t b^\delta (e^{-t} \phi_t v, e^{-t} \phi_t v).
\end{equation}
Combining (\ref{inequality sff 1}) and (\ref{inequality sff 2}) we get
\begin{equation} \label{inequality sff 3}
b_t(v,v) \geq   e^t b^\delta (e^{-t} \phi_t v, e^{-t} \phi_t v).
\end{equation}
%$$
%\frac{b_t(v,v)}{\phi_t^2\delta(v,v)} \geq \phi_t^{-1} e^t\frac{b^\delta (v,v)}{\delta(v,v)}.
%$$
If $g_t$ and $g^\delta$ denote the metrics of $\Sigma_t$ and $\Sigma^\delta$, respectively, one easily checks that
\begin{equation} \label{relation between metrics}
    g_t(v,v) = g^\delta(e^{-t}\phi_t v,e^{-t}\phi_t v).
\end{equation}
Thus, (\ref{inequality sff 3}) and (\ref{relation between metrics}) yields
$$
\frac{b_t(v,v)}{g_t(v,v)} \geq \phi_t^{-1} e^t\frac{b^\delta (e^{-t}\phi_t v,e^{-t}\phi_t v)}{g^\delta(e^{-t}\phi_t v,e^{-t}\phi_t v)},
$$
for any tangent vector $v$.
Therefore, since $\phi_t$ converges uniformly to $2$ as $t$ goes to infinity and
$\Sigma^\delta$ is strictly convex,
%as an Euclidean hypersurface,
we can choose $T \in [0, \infty)$ so that all of the principal curvatures of $\Sigma_t$ are no less than $1$, for each $t \geq T$.

\end{proof}
\section{Proofs of the theorems}

We begin with the proof of Theorem \ref{second theorem}. Let $\Sigma$ be a star-shaped hypersurface in $\mathbb{H}^n$ whose mean curvature satisfies $H_1 \geq 1$. Then $\Sigma^\delta$ is a star-shaped hypersurface in $\mathbb{R}^n$. Moreover, by Proposition \ref{mean curvatures}, $\Sigma^\delta$ is strictly mean-convex. 
%{\color{red}Then $\Sigma$ can also be seen as a star-shaped hypersurface in Euclidean space (just consider Poincar\'e's model for $\mathbb{H}^n$ and replace the hyperbolic metric by the Euclidean metric $\delta$).} 
By a result proved in \cite{GR} it follows that
\begin{equation} \label{desigualdade GR}
\mathcal{Q}(\Sigma^\delta) > \left( \frac{n-1}{n} \right)^2.
\end{equation}
Let $\Sigma_t$, with $\Sigma_0 = \Sigma$, be the one-parameter family of hypersurfaces defined by (\ref{varphi_t}).
By Proposition \ref{lim P to Q}  and (\ref{desigualdade GR}) we have
$$
\lim_{t \to \infty} \mathcal{P}(\Sigma_t) > \left( \frac{n-1}{n} \right)^2.
$$
Since, by Proposition \ref{prop P noinicreasing 2}, $\mathcal{P}(\Sigma_t)$ is nonincreasing, we conclude that
$$
\mathcal{P}(\Sigma_0) > \left( \frac{n-1}{n} \right)^2,
$$
which is just a rewriting of (\ref{weak inequality}).

\begin{remark} {\rm
The quantities $\mathcal{P}(\Sigma)$ and $\mathcal{Q}(\Sigma)$ also make sense when $n=2$. Moreover, it is known that if $\Sigma \subset \mathbb{R}^2$ is convex, then
$$
\mathcal{Q}(\Sigma) > \frac{(2\pi)^2}{54}
$$
(see \cite{Sachs1, Sachs2, Hall}). 
Thus, by proceeding as above, one can show that if $\Sigma$ is a hypersurface in $\mathbb{H}^2$ satisfying
\begin{equation} \label{kappa geq 1}
\kappa \geq 1,
\end{equation}
where $\kappa$ denotes the geodesic curvature of $\Sigma$, then it holds that
$$
\int_\Sigma |x|^2 \, d\Sigma > 2\pi \left[ \frac{(2\pi)^2}{54} \left( \frac{|\Sigma|}{2\pi} \right)^4 + \left( \frac{|\Sigma|}{2\pi} \right)^2 \right]^{\frac{1}{2}},
$$
that is, 
$$
\int_\Sigma |x|^2 \, d\Sigma > |\Sigma|\left( \frac{1}{54}|\Sigma|^2 + 1  \right)^{\frac{1}{2}}.
$$
Also, by considering a sequence $\lbrace \Lambda_n \rbrace$ of convex curves in $\mathbb{R}^2$ that converges, in the $C^0$ topology, to an equilateral triangle centered at the origin, one can show, by suitably rescaling the terms of $\lbrace \Lambda_n \rbrace$, that $1/54$ is the largest constant $\Theta$ for which the inequality
$$
\int_\Sigma |x|^2 \, d\Sigma > |\Sigma|\left( \Theta|\Sigma|^2 + 1  \right)^{\frac{1}{2}}
$$
holds for every hypersurface in $\mathbb{H}^2$ satisfying (\ref{kappa geq 1}). We leave the details to the interested reader.}
\end{remark}

Now let us prove Theorem \ref{first theorem}. It is proved in \cite{GR} that there exists a strictly convex surface $\Gamma^\delta$ in $\mathbb{R}^3$ such that
\begin{equation} \label{Q < 1}
\mathcal{Q}(\Gamma^\delta) < 1.
\end{equation}
By the scale invariance of $\mathcal{Q}$, we can assume that $\Gamma^\delta \subset \mathbb{B}^3$. Denote by $\Gamma$ the surface $\Gamma^\delta$ when seen as a hypersurface in $\mathbb{H}^3$. Let $\Gamma_t$, with $\Gamma_0 = \Gamma$, be defined as in (\ref{varphi_t}). Inequality (\ref{Q < 1}) together with Proposition \ref{lim P to Q} give
$$
\lim_{t \to \infty} \mathcal{P}(\Gamma_t) < 1.
$$
Thus, there exists $t_0 \in [0,\infty)$ for which $\mathcal{P}(\Gamma_{t}) < 1$, for all $t \geq t_0$. To finish the proof, notice that Proposition \ref{c to hc} guarantees that $t_0$ can be chosen so that $\Gamma_t$ is horospherically convex, for each $t \geq t_0$.

Next, let us prove Theorem \ref{third theorem}. The proof consists of an induction argument very similar to the one given in \cite{GWW2}, but with (\ref{weak inequality}) as the base case.

The $k=0$ case follows from Theorem \ref{second theorem}.

Let $j$ be an integer such that $2j \in \lbrace 0, 1, \ldots, n-3 \rbrace$ and suppose that the inequality holds for $k=2j$, that is, suppose 

%$$
%\int_\Sigma \rho H_{2j} \, d\Sigma >
%\omega_{n-1} \left[ \left( \frac{n-1}{n} \right)^2 \left( \frac{|\Sigma|}{\omega_{n-1}} \right)^{\frac{2n}{(2j+1)(n-1)}} + \left( \frac{|\Sigma|}{\omega_{n-1}} \right)^{\frac{2(n-2j-1)}{(2j+1)(n-1)}} \right]^{\frac{2j+1}{2}}.
%$$
%bla
$$
\int_\Sigma \rho H_{2j} \, d\Sigma >
\omega_{n-1} \left( \frac{|\Sigma|}{\omega_{n-1}}\right)^{\frac{n}{n-1}}\left[\left( \frac{n-1}{n}\right)^2 + \left( \frac{|\Sigma|}{\omega_{n-1}}\right)^{-\frac{2}{n-1}}\right]^{\frac{2j+1}{2}}.
$$
It was proved in \cite{GWW3} (see also \cite{GWW4} and \cite{LWX2}) that
\begin{equation} \label{Wang and Xia}
\int_{\Sigma} H_{2j+2} \, d\Sigma \geq |\Sigma| \left[ 1 + \left( \frac{|\Sigma|}{\omega_{n-1}}\right)^{-\frac{2}{n-1}} \right]^{j+1}.
\end{equation}
H\"older's inequality and (\ref{Wang and Xia}) give
\begin{align*}
\left( \int_\Sigma \rho H_{2j+2} \, d\Sigma \right) \left( \frac{H_{2j+2}}{\rho} \, d\Sigma \right)    & \geq  \left( \int_\Sigma H_{2k+2} \, d\Sigma \right)^2 \\
& \geq  |\Sigma|^2 \left[ 1 + \left( \frac{|\Sigma|}{\omega_{n-1}}\right)^{-\frac{2}{n-1}} \right]^{2(j+1)} \\
& > |\Sigma|^2 \left[ \left( \frac{n-1}{n} \right)^2 + \left( \frac{|\Sigma|}{\omega_{n-1}}\right)^{-\frac{2}{n-1}} \right]^{2(j+1)}.
\end{align*}
Thus, if we set
$$
\alpha = |\Sigma|^2 \left[ \left( \frac{n-1}{n} \right)^2 + \left( \frac{|\Sigma|}{\omega_{n-1}}\right)^{-\frac{2}{n-1}} \right]^{2(j+1)},
$$
we find that
\begin{equation} \label{9.2}
\int_\Sigma \rho H_{2j+2} \, d\Sigma - \int_\Sigma \frac{H_{2j+2}}{\rho} \, d\Sigma < \int_\Sigma \rho H_{2j+2} \, d\Sigma -\frac{\alpha}{\int_\Sigma \rho H_{2j+2} \, d\Sigma}.
\end{equation}

It is also known (see \cite{GWW2}, Theorem 8.1) that
\begin{equation} \label{crucial}
\int_\Sigma \rho H_{2j+2} \, d\Sigma - \int_\Sigma \frac{H_{2j+2}}{\rho} \, d\Sigma \geq \int_\Sigma \rho H_{2j} \, d\Sigma.
\end{equation}
Hence, from (\ref{crucial}) and the induction hypothesis, we find
\begin{align}
%\begin{equation}
\label{crucial + hypothesis}
& \int_\Sigma \rho H_{2j+2} \, d\Sigma - \int_\Sigma \frac{H_{2j+2}}{\rho} \, d\Sigma \nonumber \\
& \quad > \omega_{n-1} \left( \frac{|\Sigma|}{\omega_{n-1}}\right)^{\frac{n}{n-1}}\left[\left( \frac{n-1}{n}\right)^2 + \left( \frac{|\Sigma|}{\omega_{n-1}}\right)^{-\frac{2}{n-1}}\right]^{\frac{2j+1}{2}}.
%\end{equation}
\end{align}

Consider the function
$f(t) = t - \alpha t^{-1}$. From (\ref{9.2}) and (\ref{crucial + hypothesis}) we have
\begin{align}
& f\left(\int_\Sigma \rho H_{2j+2} \, d\Sigma \right) \nonumber \\
& \quad > \omega_{n-1} \left( \frac{|\Sigma|}{\omega_{n-1}}\right)^{\frac{n}{n-1}}\left[\left( \frac{n-1}{n}\right)^2 + \left( \frac{|\Sigma|}{\omega_{n-1}}\right)^{-\frac{2}{n-1}}\right]^{\frac{2j+1}{2}}. \label{9.4}
\end{align}
We also have
\begin{align*}
 f&\left( \omega_{n-1} \left( \frac{|\Sigma|}{\omega_{n-1}}\right)^{\frac{n}{n-1}}\left[\left( \frac{n-1}{n}\right)^2 + \left( \frac{|\Sigma|}{\omega_{n-1}}\right)^{-\frac{2}{n-1}}\right]^{\frac{2j+3}{2}} \right) \\
& = \left( \frac{n-1}{n} \right)^2 \omega_{n-1} \left( \frac{|\Sigma|}{\omega_{n-1}}\right)^{\frac{n}{n-1}}\left[\left( \frac{n-1}{n}\right)^2 + \left( \frac{|\Sigma|}{\omega_{n-1}}\right)^{-\frac{2}{n-1}}\right]^{\frac{2j+1}{2}} \\
& < f\left(\int_\Sigma \rho H_{2j+2} \, d\Sigma \right),
\end{align*}
where the last inequality follows from (\ref{9.4}). Since $f$ is increasing on $[0,\infty)$, we find that
$$
\int_\Sigma \rho H_{2j+2} \, d\Sigma > 
\omega_{n-1} \left( \frac{|\Sigma|}{\omega_{n-1}}\right)^{\frac{n}{n-1}}\left[\left( \frac{n-1}{n}\right)^2 + \left( \frac{|\Sigma|}{\omega_{n-1}}\right)^{-\frac{2}{n-1}}\right]^{\frac{2j+3}{2}},
$$
which completes the induction.
%We continue to use Poincar\'e's model for $\mathbb{H}^n$, so that any hypersurface $\Sigma$ in $\mathbb{H}^n$ can also be seen, by considering the flat metric on $\mathbb{B}^n$, as an Euclidean hypersurface. We denote by $H_1$ the mean curvature of $\Sigma$ as a hypersurface in $\mathbb{H}^n$ and by $H_1^\delta$ the mean curvature of $\Sigma$ as a hypersurface in $\mathbb{B}^n$.

%\begin{remark}
%{\rm Notice that $\Sigma_t^\delta$ is just the image of $\Sigma_0$ under the homothety of center in the origin and ratio $e^{-t}$.}
%\end{remark}

%\bibliographystyle{abbrv}
%\bibliography{ref}

\end{document}